\documentclass[12pt, reqno]{amsart}
\allowdisplaybreaks[1]
\usepackage{amsmath}
\usepackage{amssymb}
\usepackage{amsfonts}
\usepackage{verbatim}
\usepackage{float}

\usepackage[usenames]{color}
 \usepackage{tikz}
\usepackage{hyperref}
\makeindex

 \newtheorem{theorem}{Theorem}[section]
 \newtheorem{corollary}[theorem]{Corollary}
  \newtheorem{conjecture}[theorem]{Conjecture}
 \newtheorem{lemma}[theorem]{Lemma}

\theoremstyle{definition}

\theoremstyle{remark}

\newtheorem{fact*}{Fact}

\newcommand{\til}{\raise.17ex\hbox{$\scriptstyle\mathtt{\sim}$}}

\newcommand\beq{\begin{equation}}

\newcommand\eeq{\end{equation}}

\newcommand\bbm{\begin{bmatrix}}
\newcommand\ebm{\end{bmatrix}}
\newcommand\bpm{\begin{pmatrix}}
\newcommand\epm{\end{pmatrix}}
\numberwithin{equation}{section}

\newlength{\Mheight}
\newlength{\cwidth}

\newcommand{\dfn}[1]{{\bf #1}\index{#1}}

\title[The wedge-of-the-edge theorem]{The wedge-of-the-edge theorem: edge-of-the-wedge type phenomenon within the common real boundary}
\author{
	J. E. Pascoe
}
\address{
J. E. Pascoe \\	
	Department of Mathematics\\
  Washington University in St. Louis\\
  One Brookings Drive \\
 St. Louis, MO 63130}
\email[J. E. Pascoe]{pascoej@math.wustl.edu}
\thanks{
Partially supported by National Science Foundation Mathematical
Science Postdoctoral Research Fellowship  
DMS 1606260}
\date{\today}


\begin{document}

\begin{abstract}
The edge-of-the-wedge theorem in several complex variables gives the analytic continuation of functions defined on the poly upper half plane and the poly lower half plane, the set of points in $\mathbb{C}^d$ with all coordinates in the upper and lower half planes respectively, through a set in real space, $\mathbb{R}^d.$
The geometry of the set in the real space  can force the function to analytically continue within the boundary itself, which is qualified in our wedge-of-the-edge theorem. For example, if a function extends to the union of two cubes in $\mathbb{R}^d$ which are positively oriented, with some small overlap, the functions must analytically continue to a neighborhood of that overlap of a fixed size not depending of the size of the overlap.
\end{abstract}
\maketitle


\section{Introduction}
Let $\Pi$ denote the open upper half plane in $\mathbb{C}.$
That is, 
$$\Pi = \{z \in \mathbb{C}|  \textrm{ Im } z > 0\}.$$
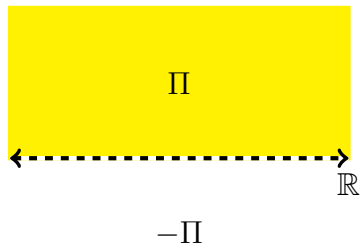
\begin{figure}[H]
\begin{tikzpicture}[xscale=.75]
		\draw [fill=lightgray,ultra thick,yellow] (0,0) rectangle (6,2);
		\draw [ultra thick,white] (0,0) -- (6,0);
		\draw [<->,dashed, ultra thick] (0,0) -- (6,0);
		\node at (3,1) {$\Pi$};
		\node at (3,-1) {$-\Pi$};
		
		\node at (6,-.35) {$\mathbb{R}$};
	
	\end{tikzpicture}
	\caption{A picture of $\Pi$.}
\end{figure}

Let $U \subseteq \mathbb{R}$ be an open set.
One can show using elementary complex analysis, for example Morera's theorem, that any continuous map
$f: \Pi \cup U \cup -\Pi \rightarrow \mathbb{C}$
which is analytic on $\Pi \cup -\Pi$ is indeed analytic on
the whole of $\Pi \cup U \cup -\Pi.$ When $f|_U$ is real-valued,
this observation is the essential part of the proof of the Schwarz reflection principle.

The analogue in several variables was proven by Bogoliubov in 1956 and is known as \emph{the edge-of-the-wedge theorem.} A quality yet concise and fully comprehensible discussion was given by Rudin in \emph{Lectures on the edge-of-the-wedge theorem}\cite{rudeow}. The edge-of-the-wedge theorem is an important stem theorem in several complex variables.
\begin{theorem}[The edge-of-the-wedge theorem]
Let $U \subseteq \mathbb{R}^n$ be an open set.
There is an open set $D$ in $\mathbb{C}^n$ containing $\Pi^n \cup U \cup -\Pi^n$
such that every function $f: \Pi^n \cup U \cup -\Pi^n \rightarrow
\mathbb{C}$ which is analytic on $\Pi^n \cup -\Pi^n$ analytically continues to $D.$
\end{theorem}
Some degustation of the edge-of-the-wedge theorem is needed before we proceed. Two some what remarkable notes are:
\begin{enumerate}
\item The set $\Pi^d\cup U \cup -\Pi^d$ is not open. That is, the open set $D$ containing $\Pi^n \cup U \cup -\Pi^n$ must intersect the exterior, somehow leaking out the sides in a big way. For example, when $n =2,$ this means
$D \cap \Pi \times -\Pi$ and $D \cap -\Pi \times \Pi$ must be nonempty. 
\begin{figure}[H]\label{figeow}
\begin{tikzpicture}[xscale=.5,yscale=.5]	
	\draw [fill, yellow] (0,0) rectangle (4,4); 
	\draw [fill, yellow] (0,0) rectangle (-4,-4); 
	\draw [dashed,ultra thick, <->] (-4,0) -- (4,0); 
	\draw [dashed,ultra thick, <->] (0,-4) -- (0,4);  
	\draw [fill] (0,0) circle [radius=0.1];
	\node [right] at (4,-.5) {$\text{Im } z$};
			\node [below right] at (0,0) {$U$};
	\node [above right] at (0,4) {$\text{Im } w$};
	\node at (2,2) {$\Pi^2$};
	\node at (-2,-2) {$-\Pi^2$};
	\node at (2,-2) {$\Pi \times -\Pi$};
	\node at (-2,2) {$-\Pi \times \Pi$};
	
	\draw [fill, yellow] (12,0) rectangle (16,4); 
	\draw [fill, yellow] (12,0) rectangle (8,-4); 
	\draw [dashed,ultra thick, <->] (8,0) -- (16,0); 
	\draw [dashed,ultra thick, <->] (12,-4) -- (12,4);  
	\draw [fill,yellow] (12,0) circle [radius=1];
	\draw [ultra thick, dashed] (12,0) circle [radius=1];
	\node at (14,2) {$\Pi^2$};
	\node at (10,-2) {$-\Pi^2$};
	\node at (14,-2) {$\Pi \times -\Pi$};
	\node at (10,2) {$-\Pi \times \Pi$};
\node at (12,0) {$D$};	
	
	\node at (6,0) {$\Rightarrow$};
	\end{tikzpicture}
	\caption{A projection of the edge-of-the-wedge theorem onto the imaginary axes. Note that all of $U$ appears as a point in this projection.}
	\end{figure}
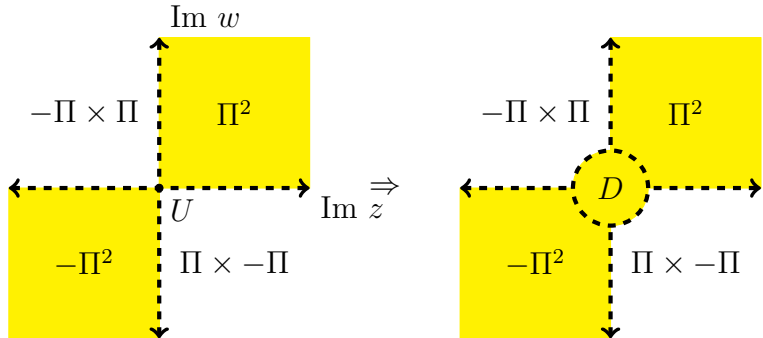
Thus, we can imagine the result is somewhat deeper than in the one variable. Additionally, we are presented with a mystery: what is the geometry of the maximal $D?$
\item The set $D$ is independent of the function $f.$ On face, this fact is somewhat surprising, but in hindsight, and in the light of Montel type theorems, how could it be otherwise.
\end{enumerate}
The metaphor evoked in the phrasing \emph{edge-of-the-wedge} merits some additional discussion. The phrase \emph{thin edge of the wedge} refers to something minor with major, often spurious, implications, evoking the classical use of a metal wedge to split logs for firewood. The continuous continuation through the superficially thin looking set $U$ between the wedges $\Pi^n$ and $-\Pi^n$ gives way to the major analytic continuation throughout the fatter set $D.$

We note that various generalizations of the edge-of-the-wedge theorem to weaker distributional notions of agreement between the two wedges, for example \cite{EPSTEIN}. Rudin's book \cite{rudeow} gives some other generalizations and some more modern surveys include\cite{RUSSSURV}. Our goals are somewhat orthogonal to these matters: during our current enterprise, we will consider the geometry of such a continuation once we already know that it exists. That is, given
$U,$ the edge-of-the-wedge theorem automatically manufactures a set $D$ and we would like to understand how $U$ shapes $D.$

 Specifically, we will be particularly interested in $D \cap \mathbb{R}^d.$ In principle, we might be tempted to assume we are forced to have $D\cap \mathbb{R}^n = U$ which is often the case-- for example, when $U$ is  a cube, that is, $U = (-1,1)^n.$
When $n=1,$ one may simply take  take $x_n$ a dense sequence in $(-1,1)^c$ and define 
$$f(z) = \sum \frac{2^{-n}}{z-x_n},$$
and for $n>1$ products of such functions will suffice.

While a general qualitative description of $D$ remains somewhat elusive, when $U$ has some nontrivial geometry in several variables, the situation is more exciting.  Under certain nice geometric conditions, specifically if $U$ is the union of two cubes with a small overlap which are positively oriented with respect to each other, we will see that our continuation is somewhat larger than $U$ within $\mathbb{R}^d.$ This is the content of our wedge-of-the-edge theorem.
\begin{theorem}[The wedge-of-the-edge theorem for hypercubes]\label{wedgecubes}
	There is an open set $D$ containing $0$ such that for any $\varepsilon>0,$
	every continuous
	function $f:\Pi^d \cup (-1,\varepsilon)^d \cup (-\varepsilon,1)^d \cup -\Pi^d \rightarrow \mathbb{C}$
	which is
 analytic on $\Pi^d\cup -\Pi^d,$
	analytically continues to $D.$
	\end{theorem}
Theorem \ref{wedgecubes}	follows from the more general Theorem \ref{wedgefull}. In fact, $D$ can be taken to contain a uniform ball around $0$ which we discuss in Subsection \ref{uniformity}.
	
		The geometric situation in the wedge-of-the-edge theorem has 
	$(-1,\varepsilon)^d$ and $(-\varepsilon,1)^d,$ pieces of the edge, emulating the roles of the wedges $\Pi^d$ and $-\Pi^d$ in the classical edge-of-the-wedge theorem. Hence the name wedge-of-the-edge theorem.
	
	\begin{figure}[H]
		\begin{tikzpicture}[xscale=.5,yscale=.5]	
	\draw [fill, yellow] (-.2,-.2) rectangle (4,4); 
	\draw [fill, yellow] (.2,.2) rectangle (-4,-4); 
	\draw [dashed,ultra thick] (-.2,-.2) rectangle (4,4); 
	\draw [dashed,ultra thick] (.2,.2) rectangle (-4,-4); 
	\draw [fill, yellow] (11.8,-.2) rectangle (16,4); 
	\draw [fill, yellow] (12.2,.2) rectangle (8,-4); 
	\draw [dashed,ultra thick] (11.8,-.2) rectangle (16,4); 
	\draw [dashed,ultra thick] (12.2,.2) rectangle (8,-4); 
	\draw [fill,yellow] (12,0) circle [radius=1];
	\draw [ultra thick, dashed] (12,0) circle [radius=1];	
	\node at (6,0) {$\Rightarrow$};
	\node at (12,0) {$D$};
	\end{tikzpicture}
	\caption{A pictographic representation of the wedge-of-the-edge theorem. Note the visual analogy with the edge-of-the-wedge theorem.}
	\end{figure}
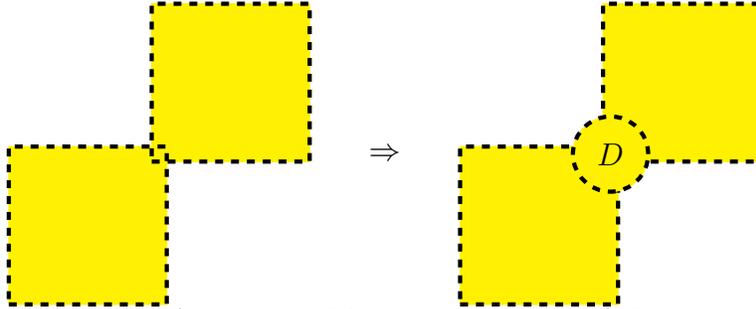

Several examples showing the necessity of the orientation and small $\varepsilon$ overlap. The function
	 $1/(1-tzw)$ can fit two squares in the opposite orientation with a small overlap in the opposite orientation as in the wedge-of-the-edge theorem, but cannot be analytically continued to a uniform neighborhood of $0$ for large $t$. The function $\sqrt{zw}$ cannot analytically continue through the point $0,$ demonstrating the necessity of the small overlap.  Our qualitative understanding is still incomplete. For example, we have not eliminated the possibility that $D$ cannot be taken to contain the whole $(-1,1)^d.$
	 \begin{figure}[H]
	 \begin{tikzpicture}[xscale=.5,yscale=.5]	
	\draw [fill, yellow] (-.2,.2) rectangle (4,-4); 
	\draw [fill, yellow] (.2,-.2) rectangle (-4,4); 
	\draw [dashed,ultra thick] (-.2,.2) rectangle (4,-4); 
	\draw [dashed,ultra thick] (.2,-.2) rectangle (-4,4); 
	\draw [very thick] (.5,4) to (.5,2) to (.75,1.3) to (1,1) to (1.3,.75) to (2,.5) to
 (4,.5);
	\draw [purple] (.5,4) to (.5,2) to (.75,1.3) to (1,1) to (1.3,.75) to (2,.5) to
 (4,.5);
 	\draw [very thick] (-.5,-4) to (-.5,-2) to (-.75,-1.3) to (-1,-1) to
 	(-1.3,-.75) to (-2,-.5) to
 (-4,-.5);
 	\draw [purple] (-.5,-4) to (-.5,-2) to (-.75,-1.3) to (-1,-1) to
 	(-1.3,-.75) to (-2,-.5) to
 (-4,-.5);
 \end{tikzpicture}
 \begin{tikzpicture}[xscale=.5,yscale=.5]	
   \draw [fill, white] (0,0) rectangle (-8,4); 
	\draw [fill, yellow] (0,0) rectangle (4,4); 
	\draw [fill, yellow] (0,0) rectangle (-4,-4); 
	\draw [ultra thick] (0,0) rectangle (4,4); 
	\draw [ultra thick] (0,0) rectangle (-4,-4); 
	\end{tikzpicture}
	\caption{Figures of the domains of $1/(1-tzw)$ and $\sqrt{zw},$ respectively. In the figure of $1/(1-tzw),$ the curve represents the singular set of the function, which as $t$ goes to infinity approaches the coordinate axes. }
	\end{figure}
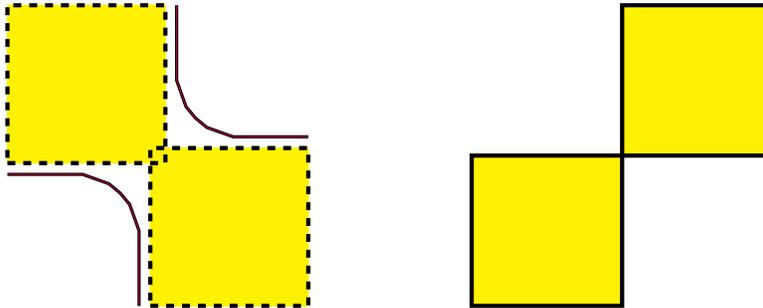

	An interesting matter to consider is the large rescaling limit of the wedge-of-the-edge theorem. That is, functions which are continuously defined on the entire positive and negative orthants plus some small overlap and the upper and lower multivariate half planes are entire. 
\begin{corollary}[The limiting wedge-of-the-edge theorem]
	For any $\varepsilon>0,$
	every continuous
	function $f:\Pi^d \cup (-\infty,\varepsilon)^d \cup (-\varepsilon,\infty)^d \cup -\Pi^d \rightarrow \mathbb{C}$
	which is
	 analytic on $\Pi^d\cup -\Pi^d,$
	analytically continues to all of $\mathbb{C}^n.$
	\end{corollary}
	
	We now give a nice, somewhat amusing corollary of the main result along the lines of Hartog's theorem. If we have a function defined on a  upper and lower multivariate upper half plane and a region in $\mathbb{R}^d$ missing a single point, we see that it must continue to that point.
	\begin{corollary}[A weak Hartog's type theorem]\label{hartogs}
	Let $E \subseteq \mathbb{R}^d$ be an open set.
	Let $p \in E.$
	If $f:\Pi^d \cup (E\setminus \{p\}) \cup -\Pi^d \rightarrow \mathbb{C}$
	which is continuous and analytic on the interior of its domain,
	then,  it has a continuous extension to $E.$	
	\end{corollary}
The following sketch gives a picture of the proof. The formal details are left to the reader.

\begin{figure}[H]
		\begin{tikzpicture}[xscale=.5,yscale=.5]	
	\draw [fill, cyan] (-4,-4) rectangle (4,4);
	\draw [fill, cyan] (8,-4) rectangle (16,4);
	\node at (-3,3) {$E$};
	\draw [fill,  blue] (-.2,-.2) rectangle (3,3); 
	\draw [fill,  blue] (.2,.2) rectangle (-3,-3); 
	\draw [dashed,ultra thick] (-.2,-.2) rectangle (3,3); 
	\draw [dashed,ultra thick] (.2,.2) rectangle (-3,-3); 
	\draw [fill, blue] (11.8,-.2) rectangle (15,3); 
	\draw [fill, blue] (12.2,.2) rectangle (9,-3); 
	\draw [dashed,ultra thick] (11.8,-.2) rectangle (15,3); 
	\draw [dashed,ultra thick] (12.2,.2) rectangle (9,-3); 
	\draw [fill,blue] (12,0) circle [radius=1];
	\draw [ultra thick, dashed] (12,0) circle [radius=1];	
	\node at (6,0) {$\Rightarrow$};
	\draw [fill, red] (.5,-.5) circle [radius=0.2];
	\draw [fill, red] (12.5,-.5) circle [radius=0.2];
	\node [below right] at (.5,-.5) {$p$};
	\node [below right] at (12.5,-.5) {$p$};
	\end{tikzpicture}
	\caption{Applying the wedge-of-the-edge theorem for hypercubes where the intersection is squeezed sufficiently near $p$ gives the proof of Corollary \ref{hartogs} since the function is forced to analytically continue to $p$.}
\end{figure}
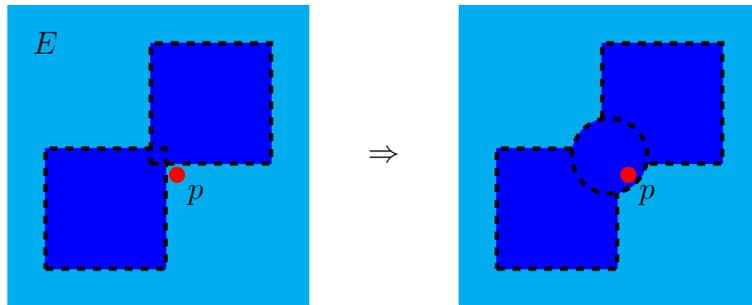

The origin of the wedge-of-the-edge theorem lies in the theory of Pick functions in several variables and their boundary values, which has received some amount of interest since the connection to multivariable operator monotonicity was established by Agler, McCarthy and Young\cite{amyloew}. Using strong aspects of their structure, a very detailed specialized version of the wedge-of-the-edge theorem was obtained by the author in \cite{pascoeBLMS} which can be used to relax the main result in \cite{amyloew} which was done in \cite{pascoeMollifier}. Some of the methods from \cite{pascoeBLMS} apply for our current endeavor, but, in general, the analysis is significantly more opaque here.

\section{The wedge-of-the-edge theorem in general}

	We define a \dfn{real wedge} $W \subseteq \mathbb{R}^d$
which is
\begin{enumerate}
\item contained within the positive orthant, that is $W \subset (\mathbb{R}^{\geq 0})^d,$
\item is a Borel set,
\item has positive measure,
\item is starlike with respect to the origin, that is, if $x \in W$ then $tx \in W$ whenever $0< t < 1.$
\end{enumerate}	

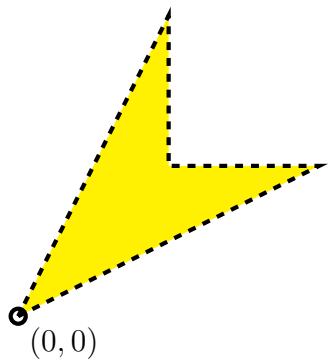
\begin{figure}[H]
\begin{tikzpicture}
\fill [yellow] (0,0) -- (4,2) -- (2,2) -- (2,4) -- (0,0);
\draw [dashed, black, ultra thick] (0,0) -- (4,2) -- (2,2) -- (2,4) -- (0,0);
\node[below right] at (0,0) {$(0,0)$} ;
\draw [black, ultra thick] (0,0) circle [radius=0.1];
\end{tikzpicture}
\caption{A real wedge.}
\end{figure}

The full version of the wedge-of-the-edge theorem is as follows.
	\begin{theorem}[The wedge-of-the-edge theorem]
	\label{wedgefull}
	Let $W$ be a real wedge.
	There is an open set $D$ containing $0$ such that for any real neighborhood $B$ of $0,$
	every continuous
	function $f:\Pi^d \cup (-1,\varepsilon)^d \cup W \cup B \cup
	-W  \cup -\Pi^d \rightarrow \mathbb{C}$
	which is  analytic on $\Pi^d\cup -\Pi^d,$
	analytically continues to $D.$
	\end{theorem}
	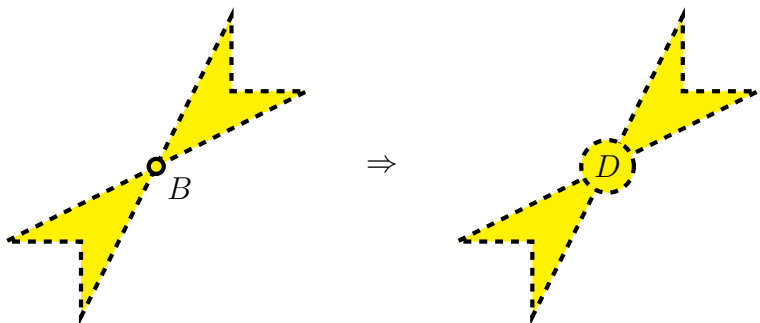
\begin{figure}[H]
\begin{tikzpicture}[xscale=.5,yscale=.5]	
\fill [yellow] (0,0) -- (4,2) -- (2,2) -- (2,4) -- (0,0);
\draw [dashed, black, ultra thick] (0,0) -- (4,2) -- (2,2) -- (2,4) -- (0,0);
\fill [yellow] (0,0) -- (-4,-2) -- (-2,-2) -- (-2,-4) -- (0,0);
\draw [dashed, black, ultra thick] (0,0) -- (-4,-2) -- (-2,-2) -- (-2,-4) -- (0,0);
\node[below right] at (0,0) {$B$} ;
\fill [yellow] (0,0) circle [radius=0.2];
\draw [black, ultra thick] (0,0) circle [radius=0.2];

\node at (6,0) {$\Rightarrow$} ;
\fill [yellow] (12,0) -- (16,2) -- (14,2) -- (14,4) -- (12,0);
\draw [dashed, black, ultra thick] (12,0) -- (16,2) -- (14,2) -- (14,4) -- (12,0);
\fill [yellow] (12,0) -- (8,-2) -- (10,-2) -- (10,-4) -- (12,0);
\draw [dashed, black, ultra thick] (12,0) -- (8,-2) -- (10,-2) -- (10,-4) -- (12,0);
\fill [yellow] (12,0) circle [radius=.7];
\draw [black, ultra thick, dashed] (12,0) circle [radius=.7];
\node at (12,0) {$D$} ;
\end{tikzpicture}
\caption{The general wedge-of-the-edge theorem.}
\end{figure}
	
	The remainder of the section will consist of a proof the wedge-of-the-edge theorem, along with some remarks on the limitations of our techniques. The trick is to bound the polynomials occurring in the power series for $f$ at $0$
	by interpolating their values after some superficially elaborate reduction.
	
	\begin{proof}
Without loss of generality our wedge will be contained inside $[0,1]^n.$	
	
	We first note that $f$ must be analytic at $0$ by the classical edge-of-the-wedge theorem.
	So, $f$ has a power series at $0.$
	Write
	\begin{equation}\label{seriesone}
	f(z)  =\sum^{\infty}_{d=0} h_d(z)
\end{equation}	 
near $0$
where $h_d$ are the homogeneous terms of the power series for $f.$
	Our goal will be to show that the series for $f$ is absolutely convergent of some neighborhood of $0$ which has size
	which depends on $W$ and not on the size of $B.$
That is it will be our goal to show that $|h_d(z)| \leq K \cdot  C^d 	\|z\|.$

Under our assumptions it is clear to see that the series in Equation \eqref{seriesone}
converges almost everywhere on $W.$
Namely, the series converges absolutely on each $(1-\varepsilon)W.$
For $x \in (1-\varepsilon)W,$ if we consider
the function
$$g(w) = f(wx) = \sum^{\infty}_{d=0} h_d(wx) = \sum^{\infty}_{d=0} h_d(x)w^d,$$
we see that the rightmost series for $g(w)$ converges on
a $\frac{1}{1-\varepsilon}\mathbb{D}$ and therefore absolutely when $w =1.$ Here, we are using the assumption that $W$ is starlike to ensure the definition of $g$ on the whole of $(-\frac{1}{1-\varepsilon},\frac{1}{1-\varepsilon})$ and since $W$ is contained in the positive orthant, $wx \in \Pi^n \cup \mathbb{R}^n \cup-\Pi^n.$

That is, it is now sufficient to prove the following lemma.
\begin{lemma} \label{lemmaone}
Fix a real wedge $W.$
There is a constant $C >0$ such that,
given $(h_d)^{\infty}_{d=0}$ a sequence of homogeneous polynomials in $n$ variables each of degree $d,$
and
$\sum^{\infty}_{d=0} h_d(x)$ converges almost everywhere on
$W$
there is a $K >0$ such that
$|h_d(z)| \leq K C^d \|z\|^d.$
\end{lemma}
	First note
	$\sum^{\infty}_{n=0} |h_n(x)|$
	converges for $x\in (1-\varepsilon)W,$
	since the function $g(w) = f(wx)$ is defined and analytic for
	$w \in \frac{1}{1-\varepsilon}\mathbb{D}.$

	Define
	$S_N = \{x \in W | \sum^{\infty}_{d=0} |h_d(x)|\leq  N\}.$
	We note that Alaoglu's theorem implies that $S_N$ is relatively closed in $W.$ That is a sequence of elements $s_i$ of
	$\ell^1(\mathbb{N})$ which converge pointwise to a limit $S$
	in $\ell^1(\mathbb{N})$ must satisfy
	$\|s\|_{\ell^1} \leq \liminf \|s_i\|_{\ell^1}.$ Here, given a sequence $x_i \in S_N$ and a limit point $x \in W$ we are taking
	$s_i = (h_d(x_i))_{d\in \mathbb{N}}$ and 
	$s =  (h_d(x))_{d\in \mathbb{N}}.$
	Namely, each $S_N$ is Borel.
	
	Moreover, the measure of  $S_N$ converges to the measure
	of $W$ as $N$ goes to infinity.
	So, there is some $N_0$ such that $S_{N_0}$ has measure greater than half the measure of $W.$
	
	Now, note that each $|h_d(x)| \leq N_0$ on $S_{N_0}.$
	So, it is now sufficient to prove the following lemma to establish Lemma \ref{lemmaone}.
	\begin{lemma}\label{lemmatwo}
Fix $n.$ Fix $p>0.$
There is a constant $C >0$ and a constant $K>0$ such that,
for each $S$ a Borel set inside $[0,1]^n$ with positive measure greater than to $p$,
given $h$ a homogeneous polynomial of degree $d$ in $n$ variables
such that $|h(x)| \leq 1$ on $S$
there is a $K >0$ such that
$|h(z)| \leq K C^d\|z\|^d.$
\end{lemma}
Now we need to introduce the $\ell^1$ norm on polynomials, which is given by the sum of the moduli of the coefficients.
That is, given a polynomial
$$h(z)  =\sum a_Iz^I,$$
We define the $\|h\|_{\ell^1}$ via the formula
$$\|h\|_{\ell^1} = \sum |a_i|.$$
Notably, if $h$ is a polynomial in some set of variables
$x_1, \ldots, x_k$ and $g$ is a polynomial in $y_1, \ldots, y_l$
we see that 
$\|hg\|_{\ell^1} = \|h\|_{\ell^1} \|g\|_{\ell^1}.$
Moreover, we note that, for homogeneous polynomials
$$|h(z)| \leq  \|h\|_{\ell^1} \| z\|^d.$$
So, now it is sufficient to show the following to establish Lemma \ref{lemmatwo}.
	\begin{lemma}
Fix $n.$ Fix $p >0.$
There is a constant $C >0$ and a constant $K>0$ such that,
for each $S$ a Borel set inside $[0,1]^n$ with positive measure greater than $p,$
given $h$ a polynomial of degree $d$ in $n$ variables
such that $|h(x)| \leq 1$ on $S$
there is a $K >0$ such that
$\|h\|_{\ell^1} \leq K C^d.$
\end{lemma}
The proof will go by induction on the number of variables.

In no variables, that is, $n=0,$ we see that $h(x) = k.$ Clearly, the maximum modulus that the constant value $k$ may have if our monomial $h$ is to be bounded on $S$ is $1.$ So we get that $\|h\|_{\ell^1} \leq 1.$

\begin{figure}[H]
\begin{tikzpicture}

\draw [dashed,ultra thick] (0,0) rectangle (4,4); %

\fill [yellow] (0.5,.2) -- (.5,3) -- (1.8,3) -- (2.6,1) -- (3.7,3.5)
-- (3.9,.5) -- (3.8,.2) -- (0.5,.2);

\draw [dashed] (0.5,.2) -- (.5,3) -- (1.8,3) -- (2.6,1) -- (3.7,3.5)
-- (3.9,.5) -- (3.8,.2) -- (0.5,.2);

\draw[ultra thick] (.6,.2) -- (.6,3);
\node[right] at (.6,2) {$P_0$};

\draw[ultra thick] (1.6,.2) -- (1.6,3);
\node[right] at (1.6,2) {$P_1$};

\draw[ultra thick] (3.6,.2) -- (3.6,3.2);
\node[left] at (3.6,2) {$P_2$};

\node[below right] at (2,1) {$S$};

\draw [fill, black] (.6,0) circle [radius=0.1];
\draw [fill, black] (1.6,0) circle [radius=0.1];
\draw [fill, black] (3.6,0) circle [radius=0.1];
\node[below right] at (.6,0) {$x_0$};
\node[below right] at (1.6,0) {$x_0$};
\node[below right] at (3.6,0) {$x_0$};
\end{tikzpicture}
\caption{The idea of the proof is to approximate on one lower dimensional pieces which are relatively large and relatively far apart, where relatively large and far apart are constrained by the measure of $S.$}
\end{figure}
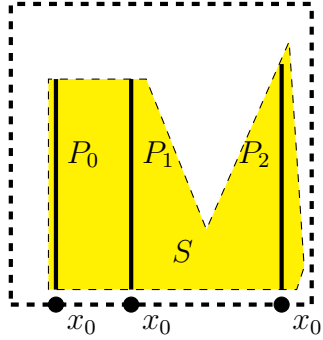

Now fix a general number of variables $n.$
One may pick points $x_0, \ldots, x_d$ such that if we define
$P_i = \{x_i\}\times [0,1]^{n-1} \cap S$
each $P_i$ has $n-1$ dimensional measure greater than $p$ and each $x_i$ is at least $p/n$ apart. The worst case is when all the measure is on one side of the box. We note that the Borel nature of $S$ ensures each slice will be Borel and hence measurable.

Let $h_i$ be a polynomial in $n-1$ variables which gives $h$ on each $P_i.$
Now, by induction we see that there is $\hat{K}$ and $\hat{C}$
such that 
$\| h_i\|_{\ell^1} \leq \hat{K} \hat{C}^{d}.$

Now
apply the Lagrange interpolation theorem, noting that interpolation recovers the polynomial exactly since we have enough nodes, to get that
$$h = \sum^d_{i=0}
		h_i\prod_{i\neq j} \frac{x-x_j}{x_i-x_j}.$$
So we see that
\begin{align*}
\|h\|_{\ell^1}& = \left\|\sum^d_{i=0}
		h_i\prod_{i\neq j} \frac{x-x_j}{x_i-x_j}\right\|_{\ell^1}, \\
& =\sum^d_{i=0}
		\|h_i\|_{\ell^1}\left\|\prod_{i\neq j} \frac{x-x_j}{p}\right\|_{\ell^1}, \\
& \leq \sum^d_{i=0}
		\|h_i\|_{\ell^1}\prod_{i\neq j} \frac{2}{(i-j)p/d}, \\
& \leq \sum^d_{i=0}
		\|h_i\|_{\ell^1}\frac{d^d}{d!p^d }\binom{d}{i}2^d, \\
& \leq \sum^d_{i=0}
		\|h_i\|_{\ell^1}\frac{d^{d}}{d!p^d d}\binom{d}{i}2^d. \\
& \leq \hat{K} \hat{C}^{d}\frac{d^{d}}{d!p^d}4^d. \\
& \leq K C^d. \\
\end{align*}
Above the in the last line we needed to apply Stirling's estimate
which says that $d^d / d!$ grows like $e^d.$ This completes the proof.

\end{proof}
\subsection{Caveat emptor}
There are several reasons the reader should beware.

Firstly, our methods generate some kind of abstract estimates on $D,$ but we do not really understand the polynomial estimates obtained. Secondly, even if we understood an optimal phrasing of something like Lemma \ref{lemmaone}, there are severe limits to our technique. We, apparently did not fully use the continuation to the full poly upper half plane-- in fact, not even the whole polydisk. (A modest abstract improvement can be obtained therefore by conformally mapping into the polydisk first.)  However, more importantly, there are limits to the polynomial interpolation method itself which do not seem to be optimal. For example taking a sum over homogenized Chebychev polynomials (which are all bounded by one on a certain wedge) would give some slight transformation of their generating function, namely $\frac{1-x}{1-2x+t^2}$, see \cite[page 69]{GENCHEB}, which satisfies the hypotheses of Lemma \ref{lemmaone}, but also apparently has singularities in the poly upper half plane.

That is, the polynomial value interpolation appears to be a rather naive method to approach the optimal wedge-of-the-edge theorem.


\section{Some concluding remarks and conjecture}
\subsection{On cones and some uniformity} \label{uniformity}
Some formulations of the edge-of-the-wedge theorem  state the theorem with respect to cones\cite{EPSTEIN, rudeow}. The goal now will be to discuss the context of cones and show that in that context we obtain some uniform amount of analytic continuation.

Let $C$ be an open cone in $\mathbb{R}^n$ with a distinguished element $\vec{1} \in C.$
We define a $\|x\|_C$ for $x\in \mathbb{R}^n$
by
$$\|x\|_C =  \max(\inf \{\lambda \in \mathbb{R}^{\geq 0} | \lambda \vec 1 - x \in C\},
\inf \{\lambda \in \mathbb{R}^{\geq 0} | \lambda \vec 1 + x \in 
C\}).$$
For example if $C$ is the positive orthant and 
$\vec{1}=(1,\ldots, 1),$ we recover the $\ell^{\infty}$
norm.
Another example would be to view $\mathbb{R}^{n\times n}$
as $n$ by $n$ Hermitian matrices and $\vec{1}$ to be the identity matrix which recovers the maximum modulus eigenvalue norm on Hermitian matrices.
We extend the norm to $x + iy \in \mathbb{C}^n$
via the formula $\|x+iy\|_C = \max(\|x\|,\|y\|).$
We define $\Pi_C = \mathbb{R}^n + i C.$

\begin{corollary}
Define $W = \{x \in C |\vec{1} - x\in C\}.$
There is an open set $D,$ containing a ball of some radius in
the $\|\cdot \|_C$ norm independent of $W$ such that
for any neighborhood $B$ of $0$
any continuous function 
$f: \Pi_C \cup W \cup B \cup  -W \cup -\Pi_C \rightarrow \mathbb{C}$
which is analytic on $\Pi_C \cup  -\Pi_C$ analytically continues
to $D.$
\end{corollary}
\begin{proof}
The novel part here is the uniformity. We will show that along any ray pointing out of the origin, we get some absolute length of analytic continuation. The key is to reduce to a $4$ dimensional problem.

First we note that any $z \in \mathbb{C}^n$
such that $\|z\| < 1/8$
can be written
$z=  x^+ - x^- + iy^+ - iy^-$
where each of the components are in $\overline{C}$ have norm
less than $1/4$, namely
$x^+ =  x + \|z\|, x^- = \|z\| - x$ etc.
Now,
consider $g(a,b,c,d) = f((a x^+ bx^- + cy^- + dy^+)/2)$
which defines a continuous function on
$\Pi^4 \cup (-1,0)^4\cup \tilde{B} \cup (-1,0)^4 \cup -\Pi^4$
which is analytic on $\Pi^4 \cup -\Pi^4.$ So applying the wedge of the edge theorem we get some small analytic continuation along the ray spanned by $z$ of  independent size and we are done.
\end{proof}



\subsection{Conjecture}
We note that the wedge-of-the-edge theorem is much stronger in the case of rational functions where it is the case that any analytic function on $\Pi^n \cup -\Pi^n$ which extends continuously through a neighborhood of a line segment from $(-1, \ldots, -1)$ to $(1,\ldots, 1)$
must analytically continue to the whole square $(-1,1)^n.$

Roughly speaking, fix $r(z)=p(z) / q(z)$ a rational function is $n$ variables which is reduced.
The singular set is then the set where $V = \{q(z) = 0\}.$
If $r(z)$ is analytic on $\Pi^n \cup -\Pi^n$ we know that
$V$ cannot intersect $\Pi^n \cup -\Pi^n,$ so by an argument involveing the inverse function theorem we get that $\nabla q (x)$
cannot be orthogonal to any positive directions at points $x\in V$, and therefore
$\nabla q(x)$ must have all positive coordinates. A geometric argument then gives the original claim. For a more detailed understanding of such varieties, which has been rapidly developed in recent years, see \cite{agmcdv, ams06, amhyp, kn08ua, BPS1}.

We are led to conjecture the following.
\begin{conjecture}
Let $U$ contain the line segment from $(-1, \ldots, -1)$ to $(1,\ldots, 1).$
Any continuous function $f:\Pi^n \cup U \cup -\Pi^n$
which is analytic on $\Pi^n \cup -\Pi^{n}$ analytically continues to $(-1,1)^n.$
\end{conjecture}
We note that in later parts of \cite{amyloew}, they give a stratagem to show that local matrix monotonicity implies global matrix monotonicity by rational approximation schemes which, in principle, would need the conjecture above to be true.

\printindex
\bibliography{references}

\begin{thebibliography}{10}

\bibitem{agmcdv}
J.~Agler and J.E. M\raise.45ex\hbox{c}Carthy.
\newblock Distinguished varieties.
\newblock {\em Acta Math.}, 194:133--153, 2005.

\bibitem{amhyp}
J.~Agler and J.E. M\raise.45ex\hbox{c}Carthy.
\newblock Hyperbolic algebraic and analytic curves.
\newblock {\em Indiana Math. J.}, 56(6):2899--2933, 2007.

\bibitem{ams06}
J.~Agler, J.E. M\raise.45ex\hbox{c}Carthy, and M.~Stankus.
\newblock Toral algebraic sets and function theory on polydisks.
\newblock {\em J. Geom. Anal.}, 16(4):551--562, 2006.

\bibitem{amyloew}
J.~Agler, J.E. M\raise.45ex\hbox{c}Carthy, and N.J. Young.
\newblock Operator monotone functions and {L\"owner} functions of several
  variables.
\newblock {\em Ann. of Math.}, 176:1783--1826, 2012.

\bibitem{BPS1}
Kelly Bickel, J.~E. Pascoe, and Alan Sola.
\newblock Derivatives of rational inner functions: geometry of singularities
  and integrability at the boundary.
\newblock {\em Proc. Lond. Math. Soc.}, 2017.
\newblock to appear.

\bibitem{EPSTEIN}
H.~Epstein.
\newblock Generalization of the ``edge-of-the-wedge'' theorem.
\newblock {\em J. Mathematical Phys}, 1:524--531, 1960.

\bibitem{kn08ua}
G.~Knese.
\newblock Polynomials defining distinguished varieties.
\newblock {\em Trans. Amer. Math. Soc.}, 362(11):5635--5655, 2010.

\bibitem{pascoeMollifier}
J.~E. Pascoe.
\newblock Note on {L\"owner}'s theorem in several commuting variables of
  {Agler}, {McCarthy} and {Y}oung.
\newblock submitted, 2014.

\bibitem{pascoeBLMS}
J.~E. Pascoe.
\newblock A wedge-of-the-edge theorem: analytic continuation of multivariable
  {P}ick functions in and around the boundary.
\newblock {\em Bull. Lond. Math. Soc.}, 2017.
\newblock to appear.

\bibitem{rudeow}
W.~Rudin.
\newblock {\em Lectures on the Edge-of-the-Wedge Theorem}.
\newblock AMS, Providence, 1971.

\bibitem{GENCHEB}
G.~Szeg{\H o}.
\newblock {\em Orthogonal Polynomials}.
\newblock American Mathematical Society Colloquium Publications, Providence,
  Rhode Island, 1939.

\bibitem{RUSSSURV}
V.~S. Vladimirov, V.~V. Zharinov, and A.~G. Sergeev.
\newblock Bogolyubov's edge-of-the-wedge theorem, its development and
  applications.
\newblock {\em Russian Math. Surveys}, 5:51–65, 1994.

\end{thebibliography}
\bibliographystyle{plain}

\end{document}